\pdfoutput=1
\RequirePackage{ifpdf}
\ifpdf % We are running pdfTeX in pdf mode
\documentclass[pdftex]{sigma}
\else
\documentclass{sigma}
\fi

\numberwithin{equation}{section}

\newtheorem{Theorem}{Theorem}[section]

\newtheorem{Lemma}[Theorem]{Lemma}
\newtheorem{Proposition}[Theorem]{Proposition}
\newtheorem{Question}[Theorem]{Question}
\newtheorem{Conjecture}[Theorem]{Conjecture}
 { \theoremstyle{definition}

\newtheorem{Remark}[Theorem]{Remark} }

\usepackage{mathabx}
\usepackage[all]{xy}
\usepackage{pgf,tikz}
\usetikzlibrary{arrows}
\usetikzlibrary{decorations.pathmorphing}

\def \d {{\partial}} %differentials and partials

\def \d {\mathfrak{d}}

\newcommand{\A}{\mathcal{A}}
\newcommand{\U}{\mathcal{U}}

\newcommand{\x}{\mathbf{x}}
\newcommand{\y}{\mathbf{y}}

\newcommand{\bbZ}{\mathbb{Z}}
\newcommand{\bbP}{\mathbb{P}}
\newcommand{\bbA}{\mathbb{A}}
\renewcommand{\d}{\dagger}

\begin{document}

\allowdisplaybreaks

\newcommand{\arXivNumber}{1807.03359}

\renewcommand{\thefootnote}{}

\renewcommand{\PaperNumber}{049}

\FirstPageHeading

\ShortArticleName{Reddening Sequences for Banff Quivers and the Class $\mathcal{P}$}

\ArticleName{Reddening Sequences for Banff Quivers\\ and the Class $\boldsymbol{\mathcal{P}}$\footnote{This paper is a~contribution to the Special Issue on Cluster Algebras. The full collection is available at \href{https://www.emis.de/journals/SIGMA/cluster-algebras.html}{https://www.emis.de/journals/SIGMA/cluster-algebras.html}}}

\Author{Eric BUCHER~$^\dag$ and John MACHACEK~$^\ddag$}

\AuthorNameForHeading{E.~Bucher and J.~Machacek}

\Address{$^\dag$~Department of Mathematics, Xavier University, Cincinnati, Ohio 45207, USA}
\EmailD{\href{mailto:buchere1@xavier.edu}{buchere1@xavier.edu}}

\Address{$^\ddag$~Department of Mathematics and Statistics, York University,\\
\hphantom{$^\ddag$}~Toronto, Ontario M3J 1P3, Canada}
\EmailD{\href{mailto:machacek@yorku.ca}{machacek@yorku.ca}}

\ArticleDates{Received June 15, 2019, in final form May 23, 2020; Published online June 08, 2020}

\Abstract{We show that a reddening sequence exists for any quiver which is Banff. Our proof is combinatorial and relies on the triangular extension construction for quivers. The other facts needed are that the existence of a reddening sequence is mutation invariant and passes to induced subquivers. Banff quivers define locally acyclic cluster algebras which are known to coincide with their upper cluster algebras.
The existence of reddening sequences for these quivers is consistent with a conjectural relationship between the existence of a~reddening sequence and a cluster algebra's equality with its upper cluster algebra.
Our result completes a~verification of the conjecture for Banff quivers. We also prove that a~certain subclass of quivers within the class $\mathcal{P}$ define locally acyclic cluster algebras.}

\Keywords{cluster algebras; quiver mutation; reddening sequences}

\Classification{13F60; 16G20}

\renewcommand{\thefootnote}{\arabic{footnote}}
\setcounter{footnote}{0}

\section{Introduction}

Cluster algebras are commutative algebras where generators can be explicitly described through a process known as quiver mutation.
It is natural to ask what influence the combinatorics of the quiver has on the algebra.
It has been observed that the existence of a maximal green or reddening sequence for a quiver seems to correspond to when the cluster algebra defined by the quiver equals its upper cluster algebra.
We provide further evidence for this relationship and aim to clarify why the two notions seem to coincide.
Our methods look at combinatorial properties of quiver and use them to either produce reddening sequences or show the cluster algebra equals its upper cluster algebra.
Moreover, we find that the same combinatorial construction, which is known as a triangular extension, is essential to our results on both reddening sequences and upper cluster algebras.

For a semifield $\mathbb{P}$ with group algebra $\mathbb{ZP}$ we choose some ground ring $\mathbb{A}$ inside the field of fractions of $\mathbb{ZP}$.
We will restrict our attention to skew-symmetric cluster algebras.
In~\cite{CAI} Fomin and Zelevinsky define a cluster algebra $\mathcal{A}$ as a certain $\mathbb{A}$-algebra inside an ambient field with generators determined by a quiver $Q$.
In~\cite{CAIII} Bernstein, Fomin, and Zelevinsky define the upper cluster algebra $\mathcal{U}$ which satisfies $\mathcal{A} \subseteq \mathcal{U}$.
A fundamental problem in cluster algebra theory is to determine when there is the equality $\mathcal{A} = \mathcal{U}$.
In general this can be a difficult problem.
Typically the ground ring of choice is $\bbA = \bbZ\bbP$.
However, there has been recent attention paid to the choice of ground ring in the $\A = \U$ question.
There can be a delicate dependence on the ground ring $\mathbb{A}$ as demonstrated by the authors with M. Shapiro~\cite{BMS}.
Goodearl and Yakimov have developed techniques for showing $\A = \U$ for ground rings $\bbA \subseteq \bbZ\bbP$~\cite{GY}.

In the case that ground ring is $\mathbb{A} = \mathbb{ZP}$, Muller's theory of cluster localization and locally acyclic cluster algebras provides a means of showing $\mathcal{A} = \mathcal{U}$~\cite{MullerADV, MullerSIGMA}.
The Banff algorithm presented by G. Muller, \cite{MullerADV}, is one way of showing that a cluster algebra is locally acyclic, and quivers for which the Banff algorithm produces a positive output are called Banff.
In particular if a quiver is Banff, then the cluster algebra it defines over $\mathbb{ZP}$ satisfies $\mathcal{A} = \mathcal{U}$.

Keller~\cite{Keller1,Keller2} introduced certain sequences of quiver mutations, which are now known as maximal green sequences and reddening sequences, as a combinatorial way to study Kontsevich and Soibelman's Donaldson-Thomas transformations~\cite{KS}.
The existence of both maximal green sequences and reddening sequences are important in cluster algebra theory and has been thought to be related to the equality of the cluster algebra and upper cluster algebra.\footnote{In~\cite{Mills} Mills states ``The conjecture that the cluster algebra and upper cluster algebra coincide if and only if a
maximal green sequence exists arose from a discussion led by Arkady Berenstein and Christof Geiss at the \emph{Hall and cluster algebras} conference at Centre de Recherches Math\'ematiques,
Montr\'eal in 2014''.}
Canakci, Lee, and Schiffler~\cite{CLS} had observed existence of a maximal green sequence coincided with $\mathcal{A} = \mathcal{U}$ in known cases at the time.
Mills~\cite{Mills} offers an explicit conjecture on the potential relationship of maximal green sequences, the equality $\mathcal{A} = \mathcal{U}$, locally acyclicity, and choice of ground ring.
Our work here is progress in understanding this relationship.

In Theorem~\ref{thm:Banff} we show that Banff quivers admit reddening sequences.
Theorem~\ref{thm:P} gives reddening sequences for any quiver in the class $\mathcal{P}$.
These theorems cover many examples which have been the subject of previous research.
Ford and Serhiyenko~\cite{FS} have shown that the quivers arising from Postnikov's reduced plabic graphs~\cite{Pos} admit reddening sequences.
Theorem~\ref{thm:Banff}, combined with the work of Muller and Speyer's~\cite{MS} which says such quivers are Banff, gives another proof of this fact.
Many quivers from marked surfaces are known to be Banff~\cite[Theorem 10.6]{MullerADV}~\cite[Proposition 12]{CLS}.
Thus, Theorem~\ref{thm:Banff} gives reddening sequences for these quivers.
For surfaces covered by these cases maximal green sequences have been previously constructed~\cite{Alim,Bucher,BucherMills,GS}.
Quivers associated to surfaces belong to an important class of quivers known as mutation finite quivers.
For mutation finite quivers there is a complete classification of the existence of reddening and maximal green sequences~\cite{MillsMGS}.
Also, minimal mutation infinite quivers are further examples of Banff quivers for which Lawson and Mills have shown have maximal green sequences~\cite{LawsonMills}.

Morally, the proposed conjecture (with some dependence on ground ring) is that for a given quiver the following are equivalent:
\begin{enumerate}\itemsep=0pt
\item[(i)] The quiver admits a reddening sequence.
\item[(ii)] The associated cluster algebra equals its upper cluster algebra.
\item[(iii)] The associated cluster algebra is locally acyclic.
\end{enumerate}

Mills \looseness=1 offers a more precise conjecture~\cite[Conjecture 2]{Mills} and verifies the conjecture for mutation finite quivers~\cite[Theorem 1.2]{Mills}.
Theorem~\ref{thm:Banff} completes a proof that all three conditions are equivalent for Banff quivers.
Theorem~\ref{thm:P} suggests the class $\mathcal{P}$ as a next step in verifying the conjecture.
To the knowledge of the authors every known Banff quiver is also in the class $\mathcal{P}$ (in fact inside a more restrictive class we denote $\mathcal{P}'$).
We ask Questions~\ref{q:BP} and~\ref{q:B'} in attempt to better understand the relationship between Banff quivers and the class $\mathcal{P}$. In Theorem~\ref{thm:locallyacyclic} we make some progress toward the equivalence of~(i),~(ii), and~(iii) for the class $\mathcal{P}$ by showing the conditions are equivalent for a certain subfamily of qui\-vers.

\section{Quiver mutation background}\label{sec:back}

In this section we will briefly establish some of the basic definitions that we will use to prove our main results on reddening sequences.
A \emph{quiver}, $Q$, is a directed graph whose edge set contains no loops or 2-cycles.
The \emph{framed quiver} associated to~$Q$, denoted $\widehat{Q}$, is the quiver whose vertex set and edge set are the following:
\begin{gather*} V\big(\widehat{Q} \big) := V(Q) \sqcup \{ i' \,|\,i \in V(Q) \},\\
 E\big(\widehat{Q} \big) : = E(Q) \sqcup \{ i \rightarrow i' \,|\,i \in V(Q) \}.
\end{gather*}
The \emph{coframed quiver} associated to $Q$, denoted $\widecheck{Q}$, is the quiver whose vertex set and edge set are the following:
\begin{gather*} V\big(\widecheck{Q} \big) := V(Q) \sqcup \{ i' \,|\,i \in V(Q) \},\\
 E\big(\widecheck{Q} \big) : = E(Q) \sqcup \{ i' \rightarrow i \,|\,i \in V(Q) \}.
\end{gather*}
The vertices $i \in V(Q)$ are the \emph{mutable vertices} and the vertices $i'$ are the \emph{frozen vertices}.
Mutation is not allowed at any frozen vertex.
The framed quiver corresponds to considering a~cluster algebra with principle coefficients.
For any mutable vertex~$i$, \emph{mutation} at the vertex~$i$ produces a new quiver denoted~$\mu_i(Q)$ obtained from~$Q$ by doing the following:
\begin{enumerate}\itemsep=0pt
\item[(1)] For each pair of arrows $k \to j$, $i \to k$ add an arrow $j \to k$.
\item[(2)] Reverse all arrows incident on~$i$.
\item[(3)] Delete a maximal collection of disjoint $2$-cycles and any arrows between two frozen vertices.
\end{enumerate}
Two quivers are said to be \emph{mutation equivalent} if one can be reached from the other by a~sequence of mutations.

Given any quiver $Q$ and $A \subseteq V(Q)$ we let $Q|_A$ denote the \emph{induced subquiver} which has
\begin{gather*}
V(Q|_A) = A,\\
E(Q|_A) = \big\{i \overset{\alpha}{\to} j \in E(Q)\colon i,j \in A\big\}
\end{gather*}
and is a natural restriction of $Q$ to $A$.
We will use $Q \setminus A$ to denote $Q|_{V(Q) \setminus A}$.

A mutable vertex is \emph{green} if there are no incident incoming arrows from frozen vertices.
Similarly, a mutable vertex is \emph{red} if there are no incident outgoing arrows to frozen vertices.
If we start with an initial quiver $Q$ and preform mutations at mutable vertices of the framed quiver $\widehat{Q}$, then any mutable vertex will always be either green or red.
The result is known as \emph{sign-coherence} and was established by Derksen, Weyman, and Zelevinsky~\cite{DWZ}.
Notice also that all vertices are initially green when starting with~$\widehat{Q}$.
Keller~\cite{Keller1, Keller2} has introduced the following types of sequences of mutations which will our main interest.
A sequence of mutations is called a \emph{reddening sequence} if after preforming this sequences of mutations all mutable vertices are red.
A \emph{maximal green sequence} is a~reddening sequence where each mutation occurs at a~green vertex.
By \cite[Proposition~2.10]{BDP}, after preforming a reddening sequence starting with~$\widehat{Q}$ the resulting quiver will be isomorphic to~$\widecheck{Q}$.
By definition all maximal green sequences are reddening sequences.
There are quivers for which a maximal green sequence does not exist, but a reddening sequence does~\cite{MullerEJC}.
Furthermore there are quivers for which no reddening sequence exists \cite{Seven}, green or otherwise.
If a reddening (maximal green) sequence exists for a quiver, we will say this quiver admits a reddening (maximal green) sequence.

Now we recall two results of Muller which will be needed for the proofs of our main results.
Both these results were first shown in~\cite{MullerEJC} and their proofs make use scattering diagrams which have been connected with cluster algebras by Gross, Hacking, Keel, and Kontsevich~\cite{GHKK}.
A~version of Lemma~\ref{lem:subquiver} also holds for maximal green sequences, but we will not need it.
However, Lemma~\ref{lem:mutation} is false for maximal green sequences.
The mutation invariance of reddening is essential to our results.

\begin{Lemma}[{\cite[Theorem 17]{MullerEJC}}]\label{lem:subquiver}
If a quiver $Q$ admits a reddening sequence, then any induced subquiver of $Q$ also admits a reddening sequence.
\end{Lemma}

\begin{Lemma}[{\cite[Corollary 19]{MullerEJC}}]\label{lem:mutation}
If a quiver $Q$ admits a reddening sequence, then any quiver mutation equivalent to $Q$ also admits a reddening sequence.
\end{Lemma}

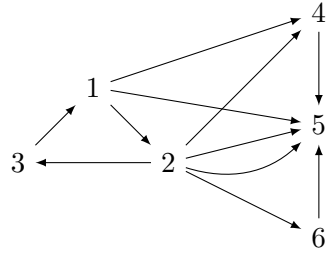
\begin{figure}[t]\centering
\begin{tikzpicture}
\node (1) at (0,0.5) {$1$};
\node (2) at (1,-0.5) {$2$};
\node (3) at (-1,-0.5) {$3$};
\draw[-{latex}] (1) to (2);
\draw[-{latex}] (2) to (3);
\draw[-{latex}] (3) to (1);
\node (4) at (3,1.5) {$4$};
\node (5) at (3,0) {$5$};
\node (6) at (3,-1.5) {$6$};
\draw[-{latex}] (4) to (5);
\draw[-{latex}] (6) to (5);
\draw[-{latex}] (1) to (4);
\draw[-{latex}] (1) to (5);
\draw[-{latex}] (2) to (4);
\draw[-{latex}] (2) to[bend right=30] (5);
\draw[-{latex}] (2) to (5);
\draw[-{latex}] (2) to (6);
\end{tikzpicture}
\caption{A triangular extension of quivers.}\label{fig:directsum}
\end{figure}

We now consider a result which states the triangular extension construction preserves the existence of reddening sequences and maximal green sequences.
Let $Q_1$ and $Q_2$ be quivers. A~\emph{triangular extension} of $Q_1$ and $Q_2$ is any quiver $Q$ with
\begin{gather*} V(Q) = V(Q_1) \sqcup V(Q_2),\\
 E(Q) = E(Q_1) \sqcup E(Q_2) \sqcup E,
\end{gather*}
where $E$ is any set of arrows such that either we have either
\[
\text{for any $i \to j \in E$ implies $i \in V(Q_1)$ and $j \in V(Q_2)$}
\]
or else
\[
\text{for any $i \to j \in E$ implies $i \in V(Q_2)$ and $j \in V(Q_1)$.}
\]
That is, a triangular extension of quivers simply takes the disjoint union of the two quivers then adds additional arrows between the quivers with the condition that all arrows are directed from one quiver to the other. An example of a triangular extension of quivers $Q_1$ and $Q_2$ where $V(Q_1) = \{1,2,3\}$ and $V(Q_2) = \{4,5,6\}$ is given in Fig.~\ref{fig:directsum}.

We shortly will state a result that the triangular extension construction preserves the existence of reddening sequences and maximal green sequences.
Garver and Musiker show an analogous result for maximal green sequences in a restricted case needed for their work which they call a~$t$-colored direct sum~\cite[Theorem 3.12]{GM} and suggest the result holds in greater generality~\cite[Remark~3.13]{GM}.
Cao and Li show the result (stated for maximal green sequences, but remark on reddening sequences~\cite[Remark 4.6]{CL}) for any triangular extension in the generality of skew-symmetrizable matrices~\cite[Theorem 4.5]{CL}.
In sections following this one we will emphasize the results for reddening sequences. The fact that the existence of reddening sequences is mutation invariant will allow applications of the following lemma to Banff quivers and quivers in class $\mathcal{P}$ since both these classes are defined up to mutation equivalence.

\begin{Lemma}[{\cite[Theorem 4.5, Remark 4.6]{CL}}]\label{lem:directsum}
If $Q_1$ and $Q_2$ are any two quivers which both admit reddening $($maximal green$)$ sequences, then any triangular extension of $Q_1$ and $Q_2$ admits a reddening $($maximal green$)$ sequence.
\end{Lemma}

\section{Existence of reddening sequences}
In this section we define Banff quivers and the class $\mathcal{P}$.
We show that Banff quivers and quivers in the class $\mathcal{P}$ admit reddening sequences.

\subsection{Banff quivers}\label{subsec:Banff}
A \emph{bi-infinite path} in a quiver~$Q$ is a sequence $(i_a)_{a \in \mathbb{Z}}$ of mutable vertices such that $i_a \to i_{a+1}$ is a arrow for each $a \in \mathbb{Z}$.
A pair of vertices $(i,j)$ is a \emph{covering pair} if $i \to j$ is an arrow which is not part of any bi-infinite path.
Muller's class of Banff quivers is the smallest class of quivers such that
\begin{itemize}\itemsep=0pt
\item any acyclic quiver is Banff,
\item any quiver mutation equivalent to a Banff quiver is Banff,
\item and any quiver $Q$ with a covering pair $(i,j)$ where both $Q \setminus \{i\}$ and $Q \setminus \{j\}$ are Banff is a~Banff quiver.
\end{itemize}
A demonstration that a quiver is Banff is shown in Fig.~\ref{fig:Banff}.

\begin{Remark}It is possible to replace ``any acyclic quiver is Banff'' with ``any quiver which is a~collection of isolated vertices is Banff'' to give closer analogy to the class $\mathcal{P}$ define in Section~\ref{subsec:P}.
\end{Remark}

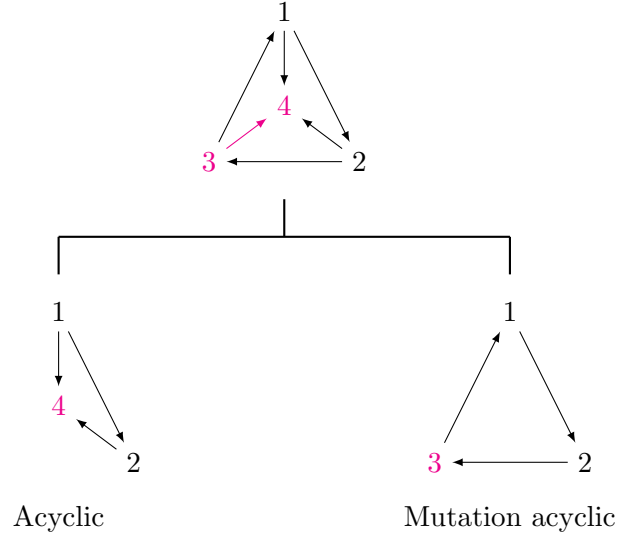
\begin{figure}[t]\centering
\begin{tikzpicture}
\node (1) at (0,1) {$1$};
\node (2) at (1,-1) {$2$};
\node (3) at (-1,-1) {$\color{magenta}3$};
\node (4) at (0,-0.25) {$\color{magenta}4$};

\draw[-{latex}] (1) to (2);
\draw[-{latex}] (2) to (3);
\draw[-{latex}] (3) to (1);
\draw[-{latex}] (1) to (4);
\draw[-{latex}] (2) to (4);
\draw[magenta, -{latex}] (3) to (4);

\draw[thick] (0,-1.5) -- (0,-2);
\draw[thick] (-3,-2) --(3,-2);
\draw[thick] (-3,-2) --(-3,-2.5);
\draw[thick] (3,-2.5) --(3,-2);

\node (1) at (-3,-3) {$1$};
\node (2) at (-2,-5) {$2$};
\node (4) at (-3,-4.25) {$\color{magenta}4$};

\draw[-{latex}] (1) to (2);
\draw[-{latex}] (1) to (4);
\draw[-{latex}] (2) to (4);

\node at (-3,-5.75) {Acyclic};

\node (1) at (3,-3) {$1$};
\node (2) at (4,-5) {$2$};
\node (3) at (2,-5) {$\color{magenta}3$};

\draw[-{latex}] (1) to (2);
\draw[-{latex}] (2) to (3);
\draw[-{latex}] (3) to (1);

\node at (3,-5.75) {Mutation acyclic};

\end{tikzpicture}
\caption{Showing a quiver is Banff using a covering pair consisting of vertices $3$ and~$4$.}\label{fig:Banff}
\end{figure}

\begin{Theorem}\label{thm:Banff}
Let $Q$ be a Banff quiver, then $Q$ admits a reddening sequence.
\end{Theorem}
\begin{proof}We will induct on the number of vertices of $Q$.
If $Q$ has a single vertex, the result is immediate.
The result is also immediate if $Q$ a collection of isolated vertices.
Now assume $Q$ is a non-isolated quiver with more than one vertex.
Since $Q$ is a Banff quiver there is covering pair in either $Q$ or some quiver mutation equivalent to $Q$.
We may assume that $Q$ has this covering pair because existence of a reddening sequence is mutation invariant by Lemma~\ref{lem:mutation}.

Let $(i,j)$ be a covering pair used in determining that $Q$ is Banff.
Let $B$ be the set of vertices~$k$ such that there exists a directed path from~$j$ to~$k$ in~$Q$.
Here we count the path consisting of no arrows, and hence $j \in B$.
Let $A$ be the complement of $B$ in $V(Q)$.
We have $i \in A$ since $(i,j)$ is a covering pair.
If $i \not\in A$, then $i \in B$ and there would be a cycle containing the arrow $i \to j$ which would contradict the fact that $(i,j)$ is a covering pair.
It follows by the definition of the sets $A$ and $B$ that $Q$ is a triangular extension of $Q|_A$ and $Q|_B$ since if there was an arrow $k \to \ell$ with $k \in B$ then $\ell \in B$.
That is, there are no arrows $k \to \ell$ with $k \in B$ and $\ell \in A$.

\looseness=1 Now $Q \setminus \{i\}$ and $Q \setminus \{j\}$ are both Banff quivers with strictly fewer vertices than $Q$.
Hence, $Q \setminus \{i\}$ and $Q \setminus \{j\}$ admit reddening sequences by induction.
The quiver $Q|_A$ is an induced subquiver of $Q \setminus \{j\}$ and $Q|_B$ is an induced subquiver of $Q \setminus \{i\}$.
Thus, $Q|_A$ and $Q|_B$ both admit reddening sequences by Lemma~\ref{lem:subquiver}.
We then conclude that $Q$ admits a reddening sequence by Lemma~\ref{lem:directsum}.
\end{proof}

\subsection[The class $\mathcal{P}$]{The class $\boldsymbol{\mathcal{P}}$}\label{subsec:P}
The following three properties define \emph{the class $\mathcal{P}$}:
\begin{itemize}\itemsep=0pt
\item The quiver with one vertex is in the class $\mathcal{P}$.
\item The class $\mathcal{P}$ is closed under quiver mutation.
\item The class $\mathcal{P}$ is closed under taking any triangular extension of two quivers in the class $\mathcal{P}$.
\end{itemize}
This class of quivers was introduced by Kontsevich and Soibelman~\cite[Section 8.4]{KS}.
The nest theorem follows immediately from results in the literature.
This theorem will be used as a~star\-ting point for exploration of classes of quivers defined by similar properties.

\begin{Theorem}\label{thm:P}
Let $Q$ be in the class $\mathcal{P}$, then $Q$ admits a reddening sequence.
\end{Theorem}
\begin{proof}The quiver with one vertex admits a reddening sequence.
The theorem then follows from Lemmas~\ref{lem:mutation} and~\ref{lem:directsum} along with the definition of the class $\mathcal{P}$.
\end{proof}

We now discuss the relationship between the class $\mathcal{P}$ and the class of Banff quivers.
Ladkani~\cite[Remark 4.21]{Lad} also considers the class $\mathcal{P}'$ of quivers defined by:
\begin{itemize}\itemsep=0pt
\item The quiver with one vertex is in the class $\mathcal{P}'$.
\item the class $\mathcal{P}'$ is closed under quiver mutation.
\item The class $\mathcal{P}'$ is closed under taking triangular extensions of quivers where one is in the class $\mathcal{P}'$ and the other is the quiver with one vertex.
\end{itemize}
Now let $\mathcal{B}'$ denote the class of quivers defined by:
\begin{itemize}\itemsep=0pt
\item Any quiver without arrows is in $\mathcal{B}'$.
\item Any quiver mutation equivalent to a quiver in $\mathcal{B}'$ is in $\mathcal{B}'$.
\item Any quiver $Q$ with an arrow $i \to j$ such that $i$ is a source or $j$ is a sink, and both $Q \setminus \{i\}$ and $Q \setminus \{j\}$ are in $\mathcal{B}'$ is in $\mathcal{B}'$.
\end{itemize}
Notice $\mathcal{P}'$ and $\mathcal{B}'$ are subclasses of $\mathcal{P}$ and $\mathcal{B}$ respectively.
These classes make use of sources (or equivalently sinks) which can be helpful.
In practice this is how quivers are often shown to be Banff.
For example, quivers from many marked surfaces~\cite[Section 10]{MullerADV} and quivers for Grassmannians~\cite{MS} are in $\mathcal{B}'$ since covering pairs are always taken to use a source (or sink).
The next proposition follows immediately from the definitions of $\mathcal{P}'$ and $\mathcal{B}'$.

\begin{Proposition}\label{prop:BP}
If a quiver is in $\mathcal{B}'$, then the quiver is in $\mathcal{P}'$.
\end{Proposition}

\begin{Question}\label{q:BP}
Does there exist a Banff quiver which is not in the class $\mathcal{P}$?
\end{Question}

\begin{Question}\label{q:B'}
Does there exist a Banff quiver which is not in $\mathcal{B}'$?
\end{Question}

A quiver has a covering pair if and only if it has a source or sink~\cite[Proposition~8.1]{MullerADV}.
So, at any step which covering pairs are available there is the option to choose a covering pair which contains a source or sink.
A quiver giving a positive answer to Question~\ref{q:B'} must fail to be Banff whenever using only source and sink covering pairs, but succeed when making use of a covering pair not containing a source nor a sink.

Lam and Speyer's class of Louise quivers~\cite{LS} is another class of quivers for which similar questions could be asked. It would be interesting to better understand the relationship of Louise quivers to the classes of quivers defined in this section.
Any Louise quiver is Banff, and to the authors' knowledge there is no known example of a Banff quiver which is not Louise.
We will not define or work further with Louise quivers here.

The results in this section readily generalize to the following system for producing quivers with reddening sequences.
We can choose a property of a quiver for which we know any quiver with this property admits a reddening sequence.
We will say a quiver is of type $T$ if it has this chosen property.
We then create a collection $\mathcal{C}$ of quivers by:
\begin{itemize}\itemsep=0pt
\item Any quiver of type $T$ is in $\mathcal{C}$.
\item Any quiver mutation equivalent to a quiver in $\mathcal{C}$ is in $\mathcal{C}$.
\item The triangular extension of any two quivers in $\mathcal{C}$ is in $\mathcal{C}$.
\end{itemize}
If a quiver is of type $T$ only if it has a single vertex, then $\mathcal{C}$ is just the class~$\mathcal{P}$. Theorem~\ref{thm:Banff} says we can take a quiver to be of type $T$ if it is Banff.
If Question~\ref{q:BP} has a positive answer then this will be a class of quivers with reddening sequences which is strictly larger than the class $\mathcal{P}$. This raises the following fairly ambitious question:

\begin{Question}\label{q:minimal generating set}
What is a minimal $T$ for which $\mathcal{C}$ consists of all quivers which admit reddening sequences?
\end{Question}

The goal would be to classify all quivers which admit a reddening sequence by looking at finding the essential generating quivers up to triangular extension and mutation. As one can see from the class $\mathcal{P}$, a small set of generating quivers can produce a large and interesting class.
Also, the collection of quivers which admit a reddening sequence is strictly larger than the class~$\mathcal{P}$.
The quivers associated to a triangulation of a torus with one boundary component and one marked point on the boundary give examples of quivers not in the class~$\mathcal{P}$ which admit reddening sequences.
Such a quiver is shown in Fig.~\ref{fig:dreaded} along with a maximal green (reddening) sequence for it.
Up to isomorphism the quiver shown in the figure is the only quiver for the torus with one boundary component and one marked point on the boundary.
Hence, it is readily checked that the quiver in Fig.~\ref{fig:dreaded} is not Banff either.
Furthermore, quivers from closed surfaces with more than one puncture admit reddening sequence~\cite{BucherMills}, but they are neither Banff~\cite[Theorem 10.9]{MullerADV} nor do they belong to class $\mathcal{P}$~\cite[Theorem~4.11]{Lad}.

\begin{figure}[t]\centering
\begin{tikzpicture}
\node (1) at (0,0) {$1$};
\node (2) at (2,0) {$2$};
\node (3) at (1,0.75) {$3$};
\node (4) at (1,2) {$4$};
\draw[-{latex}] (1) to (2);
\draw[-{latex}] (1) to (3);
\draw[-{latex}] (2) to (3);
\draw[-{latex}] (4) to (1);
\draw[-{latex}] (4) to (2);
\draw[-{latex}, bend right=13] (3) to (4);
\draw[-{latex},bend left=13] (3) to (4);
\end{tikzpicture}
\caption{Quiver for torus with one boundary component and one marked point. A maximal green sequence for this quiver is $(1,3,4,2,1,3)$.}\label{fig:dreaded}
\end{figure}
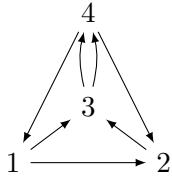

\section{Locally acyclic cluster algebras}\label{sec:cluster}
Let $\mathbb{P}$ be a semifield.
Let $\mathcal{F}$ be a field which contains $\mathbb{Z}\mathbb{P}$.
A \emph{seed} of rank $n$ in $\mathcal{F}$ is a triple $(\mathbf{x},\mathbf{y},Q)$.
The \emph{cluster} $\mathbf{x} = \{x_1,x_2, \dots , x_n\}$ is an $n$-tuple in $\mathcal{F}$ which freely generates $\mathcal{F}$ as a~field over the fraction field of $\mathbb{Z}\mathbb{P}$.
The \emph{coefficients} $\mathbf{y} = \{y_1,y_2, \dots , y_n\}$ consist of elements of $\mathbb{P}$.
Here~$Q$ a is quiver on vertex set $V(Q) = \{1,2,\dots, n\}$.

A seed $(\mathbf{x},\mathbf{y}, Q)$ may be mutated at any index $1 \leq i \leq n$, to produce a new seed $(\mu_i(\mathbf{x}),\allowbreak \mu_i(\mathbf{y}),\mu_i(Q))$.
Quiver mutation works exactly as defined in Section~\ref{sec:back}.
Let~$Q_{ij}$ denote the difference between the number of arrows $i \to j$ and the number of arrows $j \to i$ in~$Q$.
Note that then skew symmetry will follow. In other words $Q_{ji} = -Q_{ij}$.
The cluster mutates as $\mu_i(\mathbf{x}) := \{x_1,x_2,\dots,x_{i-1},x_i', x_{i+1},\dots, x_n\}$, where
\[x_ix_i':=\frac{y_i}{y_i \oplus 1} \prod_{Q_{ij} > 0} x_j^{Q_{ij}} + \frac{1}{y_i \oplus 1} \prod_{Q_{ji} >0}x_j^{Q_{ji}}.\]
The coefficients mutate by $\mu_i(\mathbf{y}) : = \{y_1',y_2', \dots, y_n'\},$ where
\[ y'_k:=
\begin{cases}
y_k^{-1} & \text{ if } k=i, \\
y_k (y_i \oplus 1)^{Q_{ik}} & \text{ if } k\neq i \text{ and } Q_{ik} \geq 0,\\
y_k \left(\dfrac{y_i}{y_i \oplus 1}\right)^{Q_{ki}} & \text{ if } k\neq i \text{ and } Q_{ki} \geq 0.
\end{cases}\]
Two seeds are call \emph{mutation equivalent} if one can be obtained from the other (up to permuting the indices) by a sequence of mutations.

Given a seed $(\mathbf{x},\mathbf{y},Q)$ we will call the union of all $\mathbf{x}'$ from any seed which is mutation equivalent to $(\mathbf{x},\mathbf{y},Q)$ the set of \emph{cluster variables}.
The \emph{cluster algebra}, $\mathcal{A} = \mathcal{A}(\mathbf{x},\mathbf{y},Q)$, is the unital $\mathbb{ZP}$-subalgebra of $\mathcal{F}$ generated by the cluster variables.
Here we have restricted our attention to the ground ring $\bbZ\bbP$ since we will make use of the theory of cluster localization.
Notice that since we are allowed to freely mutate when generating the cluster variables, that two mutation equivalent seeds will generate the same cluster algebra.

The Laurent phenomenon~\cite[Theorem 3.1]{CAI} states that $\mathcal{A}$ is a subalgebra of $\mathbb{ZP}\big[x_1^{\pm 1},x_2^{\pm 1},\dots ,\allowbreak x_n^{\pm 1}\big]$.
The \emph{upper cluster} algebra is denoted by $\mathcal{U}$ or $\mathcal{U}(\mathbf{x},\mathbf{y},Q)$ and defined by
\[ \mathcal{U} := \bigcap_{(\mathbf{x},\mathbf{y},Q)} \mathbb{ZP}\big[x_1^{\pm 1},x_2^{\pm 1},\dots , x_n^{\pm 1}\big],\]
where the intersection is taken over all seeds.
The Laurent phenomenon gives an inclusion $\mathcal{A} \subseteq \mathcal{U}$.
We will be interested in conditions on the quiver $Q$ which imply $\mathcal{A} = \mathcal{U}$.

To show $\mathcal{A} = \mathcal{U}$ we will use Muller's theory of locally acyclic cluster algebras~\cite{MullerADV}.
Let $(\mathbf{x},\mathbf{y},Q)$ be a seed of rank $n$.
The \emph{freezing} of $\A$ at $x_n \in \x$ is the cluster algebra $\A^{\dagger} = \A\big(\mathbf{x}^{\d}, \mathbf{y}^{\d}, Q^{\d}\big)$ defined as follows
\begin{itemize}\itemsep=0pt
\item The new semifield is $\bbP^{\d} = \bbP \times \bbZ$ with $x_n$ as the generator of the free abelian group $\bbZ$. The auxiliary addition is extended as
\[\big(p_1 x_n^a\big) \oplus \big(p_2 x_n^b\big) = (p_1 \oplus p_2)x_n^{\min(a,b)}.\]
\item The new ambient field is $\mathcal{F}^{\d} = \mathbb{Q}\big(\bbP^{\d}, x_1, x_2, \dots, x_{n-1}\big)$ and the new cluster is $\x^{\d} = (x_1, x_2, \dots, x_{n-1})$.
\item The new coefficients are $\y^{\d} = \big(y_1^{\d}, y_2^{\d}, \dots, y_{n-1}^{\d}\big)$ where $y_i^{\d} = y_i x_n^{Q_{in}}$.
\item The new quiver $Q^{\d}$ is obtained from $Q$ by deleting the vertex $n$.
\end{itemize}
We will denote the new upper cluster algebra by $\U^{\d}$.
By permuting indices we can freeze at any $x_i \in \x$.
Freezing at a subset of cluster variables can be done iteratively and is independent of the order of freezing.
When the freezing $\A^{\d}$ of $\A$ at $\{x_{i_1}, x_{i_2}, \dots, x_{i_m}\} \in \x$ satisfies $\A^{\d} = \A\big[(x_{i_1}x_{i_2}\cdots x_{i_m})^{-1}\big]$ we then call $\A^{\d}$ a \emph{cluster localization}.
A~\emph{cover} of $\A$ is a collection $\{\A_i\}_{i \in I}$ of cluster localizations such that there exists $i \in I$ where $\A_iP \subsetneq \A_i$ for any prime ideal $P \subseteq \A$.
A~cluster algebra is called \emph{acyclic} if it has a seed $(\mathbf{x},\mathbf{y},Q)$ where $Q$ is an acyclic quiver.
A~\emph{locally acyclic} cluster algebra is a cluster algebra for which there exists a cover by acyclic cluster algebras.
Being a cover is a transitive property.
Thus to show a cluster algebra is locally acyclic it suffices to produce a cover by cluster algebras known to be locally acyclic.
Another key property of covers is that equality of the cluster algebra with its upper cluster algebra can be checked locally.
\begin{Lemma}[{\cite[Lemma 2]{MullerSIGMA}}]\label{lem:Ucover}
Let $\{\A_i\}_{i\in I}$ be a cover of $\A$. If $\A_i=\U_i$ for all $i \in I$, then $\A = \U$.
\end{Lemma}

We also will use a property which guarantees that a freezing is a cluster localization.

\begin{Lemma}[{\cite[Lemma 1]{MullerSIGMA}}]\label{lem:freezelocal}
If $\A^{\d}$ is a freezing of a cluster algebra $\A$ at cluster variables $\{x_{i_1}, x_{i_2}, \dots, x_{i_m}\}$ and $\A^{\d} = \U^{\d}$, then $\A^{\d} = \A\big[(x_{i_1}, x_{i_2}, \dots, x_{i_m})^{-1}\big]$ is a cluster localization.
\end{Lemma}

We now prove a lemma that will give us a cover. Observe in the situation of Lemma~\ref{lem:cover} the Banff algorithm would freeze $i$ and some $j$ with $i \to j$. The only change we are making is freezing at potentially more vertices.

\begin{Lemma}\label{lem:cover} Let $(\x,\y,Q)$ be a seed for a cluster algebra $\A$ so that
\[V(Q) = \{i\} \sqcup A \sqcup B,\]
where $i$ is a source and there exists an arrow $i \to j$ for all $j \in A$.
Furthermore, let $\A^{\d}$ denote the freezing at $\{i\}$ and $\A^{\d\d}$ the freezing at~$A$.
If both $\A^{\d\d}$ and $\A^{\d}$ are cluster localizations, then $\big\{\A^{\d}, \A^{\d\d}\big\}$ is a cover for $\A$.
\end{Lemma}
\begin{proof}Under the hypothesis of the lemma it suffices to check for any prime ideal $P \subseteq \A$ either $\A^{\d}P \subsetneq \A^{\d}$ or $\A^{\d\d}P \subsetneq \A^{\d\d}$.
Consider a prime ideal $P$ and the mutation relation at $i$ which gives
\[x_{i}x'_{i} = \frac{y_iM}{y_i \oplus 1} \prod_{j \in A} x_j^{Q_{ij}} + \frac{1}{y_i \oplus 1},\]
where $M$ is a monomial in cluster variables.
Since $1/(y_i \oplus 1)$ is invertible in $\A$ and $Q_{ij} > 0$ for all $j \in A$ it follows that $1 \in \A x_{i} + \A \big(\prod_{j \in A} x_j\big)$.
Since $P$ is prime, it is a proper ideal.
So, $x_{i} \in P$ implies $x_j \not\in P$ for all $j \in A$.
If $x_{i} \not\in P$, then $\A^{\d}P \subsetneq \A^{\d}$.
If $x_{j} \not\in P$ for all $j \in A$, then $\A^{\d\d}P \subsetneq \A^{\d\d}$.
\end{proof}

Let us now consider a class of quivers we call $\mathcal{P}'_{m}$ defined by:
\begin{itemize}\itemsep=0pt
\item The quiver with one vertex is in the class $\mathcal{P}'_m$.
\item The class $\mathcal{P}'_m$ is closed under quiver mutation.
\item If $Q$ is a quiver in $\mathcal{P}'_m$, then $\mathcal{P}'_m$ contains any triangular extension of $Q$ and a quiver with a single vertex $v$ such that the number of vertices in $Q$ not connected to $v$ by an arrow is less than or equal to $m$.
\end{itemize}

\begin{Remark}It is clear that $\mathcal{P}'_{m} \subseteq \mathcal{P}'$ for any $m$.
One should note that it is indeed the case that $\mathcal{P'}_m \subsetneq \mathcal{P}'$.
Ladkani~\cite{Lad} has classified which finite mutation type quivers are in $\mathcal{P}'$.
Using this classification we can choose a~quiver $Q$ in $\mathcal{P}'$ coming from a triangulation of a surface which has $N$ vertices for some $N \geq 8$.
Since any arc in a~triangulation will be in one quadrilateral a~quiver associated to a~surface will only have vertices of degree at most $4$.
It follows that $Q$ is not in $\mathcal{P}'_m$ when $m < N-5$.
\end{Remark}

We will focus on $m=3$ and make use of the fact that quivers on $3$ vertices admit reddening sequences if and only if they are acyclic. We record this as the following lemma which is implied by a result of Seven~\cite[Theorem~1.4]{Seven} on $c$-vectors which holds in skew-symmetrizable generality.

\begin{Lemma}[{\cite[Theorem~1.4]{Seven}}]\label{lem:3vertices}
Let $Q$ be a quiver on three or fewer vertices.
The quiver $Q$ admits a reddening sequence if and only if $Q$ is mutation equivalent to an acyclic quiver.
\end{Lemma}

We are now ready to prove our main theorem on locally acyclic cluster algebras.

\begin{Theorem}\label{thm:locallyacyclic}
If $Q$ is a quiver in the class $\mathcal{P}'_3$, then for any seed $(\x,\y,Q)$ containing the quiver~$Q$ the cluster algebra $\mathcal{A}(\x,\y,Q)$ is locally acyclic.
\end{Theorem}
\begin{proof}We will induct on the number of vertices.
The theorem holds when $Q$ has one vertex.
Now assume the theorem is true for quivers in the class $\mathcal{P}'_3$ which have $n$ vertices.
Take a quiver in the class $\mathcal{P}'_3$ which has $n$ vertices and consider a triangular extension with a new vertex $i_0$.
Let this quiver be denoted $Q$ so that
\[V(Q) = \{i_0\} \sqcup A \sqcup B,\]
where $i_0$ is a source, there exists an arrow $i_0 \to j$ for all $j \in A$, and $|B| \leq 3$ such that there are no arrows from $i_0$ to any vertex in $B$.

Let $\mathcal{A} = \mathcal{A}(\x,\y,Q)$ and denote the freezings at $\{i_0\}$ and $A$ by $\mathcal{A}^{\d}$ and $\mathcal{A}^{\d\d}$ respectively.
Now~$\mathcal{A}^{\d}$ is defined by a seed with quiver $Q \setminus \{i_0\}$ which is in the class $\mathcal{P}'_3$.
Hence, $\mathcal{A}^{\d}$ is locally acyclic by induction.
This implies $\mathcal{A}^{\d}$ equals its upper cluster algebra and this is a cluster localization by Lemma~\ref{lem:freezelocal}.
The freezing $\mathcal{A}^{\d\d}$ is defined by a seed with a quiver which is the disjoint union of the vertex $i_0$ and $Q|_B$.
Since $Q$ is in $\mathcal{P}_3'$ it admits a reddening sequence by Theorem~\ref{thm:P}.
Thus $Q|_B$ admits a reddening sequence by Lemma~\ref{lem:subquiver}.
It follows by Lemma~\ref{lem:3vertices} that $Q|_B$ is mutation acyclic since $|B| \leq 3$.
Thus $\mathcal{A}^{\d\d}$ agrees with its upper cluster algebra and is a cluster localization by Lemma~\ref{lem:freezelocal}.

We can now apply Lemma~\ref{lem:cover} to conclude that $\big\{\mathcal{A}^{\d}, \mathcal{A}^{\d\d}\big\}$ is a cover of $\mathcal{A}$.
Therefore $\mathcal{A}$ is locally acyclic as desired since $\mathcal{A}^{\d}$ is locally acyclic and $\mathcal{A}^{\d\d}$ is acyclic.
\end{proof}

We believe the condition of connecting to all but at most three vertices to be artificial.
We do offer the following conjecture of a more natural result.

\begin{Conjecture}\label{conj:P'}
If $Q$ is in the class $\mathcal{P}'$, then $\mathcal{A} = \mathcal{U}$ $($over $\mathbb{ZP})$ for any cluster algebra defined by a seed with the quiver $Q$.
\end{Conjecture}

There are obstacles to extending the result in Theorem~\ref{thm:locallyacyclic} to Conjecture~\ref{conj:P'}.
The main problem is that being locally acyclic over $\bbZ\bbP$ does not pass to induced subquivers in general~\cite[Remark~3.11]{MullerADV}.

\begin{Conjecture}\label{conj:P}
If $Q$ is in the class $\mathcal{P}$, then $\mathcal{A} = \mathcal{U}$ over some ground ring $($over $\mathbb{ZP}?)$ for the cluster algebra defined by~$Q$.
\end{Conjecture}

\subsection*{Acknowledgements}

The authors wish to thank the anonymous referees for their feedback which has improved this paper.

\pdfbookmark[1]{References}{ref}
\LastPageEnding

\end{document}